\documentclass[12pt,a4paper]{amsart}

\usepackage{amsmath,amsfonts,amssymb,amsthm,mathrsfs}

\usepackage{color}

\newcommand{\R}{\mathbb{R}}

\newcommand{\un}{\mathbf{1}\!\!{\rm I}} 
\newcommand{\be}{\begin{equation}} 
\newcommand{\ee}{\end{equation}}
\newcommand{\bea}{\begin{eqnarray}} 
\newcommand{\eea}{\end{eqnarray}}
\newcommand{\bean}{\begin{eqnarray*}} 
\newcommand{\eean}{\end{eqnarray*}}
\newcommand{\rf}[1]{(\ref {#1})}

\def\dx{\,{\rm d}x}
\def\dy{\,{\rm d}y}
\def\dt{\,{\rm d}t}
\def\dr{\,{\rm d}r}
\def\dta{\,{\rm d}\tau}

\def\e{\varepsilon}

\def\s{\sigma}

\def\g{\gamma}

\def\r{\varrho}
\def\xn{|\!|\!|}

\def\mn{|\!\!|}

\def\mn2{|\!\!|_{M^{d(p-1)/2}}}

\newtheorem{theorem}{Theorem}

\newtheorem{corollary}[theorem]{Corollary}

\theoremstyle{definition}

\def\qed{\hfill $\square$}

\theoremstyle{remark}
\newtheorem{remark}[theorem]{Remark}

\author[P. Biler]{Piotr Biler}
\address{\small Instytut Matematyczny, Uniwersytet Wroc\l awski,
 pl. Grunwaldzki 2/4, \hbox{50-384} Wroc\-\l aw, Poland}
\email{Piotr.Biler@math.uni.wroc.pl}

\title[Nonlinear heat equation]{Blowup  versus global in time existence \\ of  solutions for nonlinear heat equations}

\begin{document}

\begin{abstract} 
This note is devoted to a simple proof of blowup of solutions for a nonlinear heat equation. 
The criterion for a blowup is  expressed in terms of a Morrey space norm and is in a sense complementary to conditions guaranteeing the global in time existence of solutions. 
The method goes back to H. Fujita and  extends to other nonlinear parabolic equations. 
\end{abstract}



\date{\today}

\thanks{The author, partially  supported by the NCN grant  {
2013/09/B/ST1/04412}, thanks Ignacio Guerra for interesting conversations leading to revisiting Fujita's method and Philippe Souplet for many pertinent remarks. }

\maketitle
\begin{center}
\emph{In memory of  Marek Burnat}
\medskip

\tiny{submitted to the special volume of {\em Topological Methods in Nonlinear Analysis}}
\bigskip
\bigskip

\end{center}

\baselineskip=18pt

\noindent {\bf Introduction}

\noindent 
We consider in this paper the Cauchy problem for the simplest example of semilinear parabolic equation in $\R^d$, $d\ge 1$, $p>1$,  
\bea
u_t&=&\Delta u+|u|^{p-1}u,\ \ x\in {\mathbb R}^d,\ t>0,\label{nlh}\\ 
u(x,0)&=&u_0(x)\label{ini}.
\eea
This problem has been thoroughly studied beginning with \cite{Fu,GK1,GK2}, and many fine properties of its solutions are known. 
For the reference, see the extensive monograph \cite{QS} and, in particular, a recent paper \cite{S}.

The purpose of this note is to give a short proof of nonexistence of global in time (and sometimes also local in time) positive solutions to problem \rf{nlh}--\rf{ini} together with criteria for blowup expressed in terms of the norms of critical Morrey spaces.  
Those criteria are, in a sense, complementary to assumptions guaranteeing the global in time existence of solutions. 
The idea of the proof goes back to the seminal paper \cite{Fu} but the straightforward connection to the Morrey spaces norms seems to be missed out  
up to now. 
An idea in \cite[proof of Th. 3]{BCKSV} in the context of radial solutions of the Keller-Segel model of chemotaxis is also reminiscent of that. 
Recently, this approach has been systematically developed for (radially symmetric solutions of) the chemotaxis models with different diffusion operators in \cite{BKZ3} as an alternative  to other proofs based on considerations of (new) moments in \cite{BK-JEE,BKZ,BKZ-NHM,BCKZ,BKZ2}. 
Here, for the nonlinear heat equation generally we do not use geometric assumptions on solutions such as radial symmetry. 
This method classically introduced by \cite{Fu} is flexible enough, and extends also to other nonlinear parabolic problems, see remarks at the end of this paper. 
\bigskip
 
 Whenever   problem \rf{nlh}--\rf{ini} admits a singular stationary solution, this  plays an important role in determining when solutions with $u_0$ featuring  singularities either lead to global in time solutions or they blow up in a finite time. 
 The form of this solution is well known, we recall this below.

\begin{theorem}[Singular stationary solutions]\label{sing-sol}
For $d\ge 3$ and $p>\frac{d}{d-2}$ the function 
\be
u_C(x)=c|x|^{-\frac{2}{p-1}}\label{sing}
\ee
with the constant 
\be
c=\left( \frac{2}{p-1}\left(d-\frac{2p}{p-1}\right)\right)^{\frac{1}{p-1}}
\ee 
is a stationary positive (weak and pointwise) solution of equation \rf{nlh}. 
\end{theorem} 

\begin{proof}
The exponent $\g=\frac{2}{p-1}$ is uniquely determined by the requirement $\g+2=p\g$, and the constant $c$ is determined by the relation 
\be 
c^{p-1}=\g(d-p\g)=\g(d-2-\g).\label{c=}
\ee
 Since $p\g<d$, this is a distributional solution of equation \rf{nlh}. 
\end{proof}

\bigskip

A typical example of a 
global in time existence result is the following 

\begin{theorem}\label{global}
Suppose that $d\ge 3$, $p>\frac{d}{d-2}$, and $u_0$ satisfies the estimate 
\be
0\le u_0(x)< \delta u_C(x) = \delta c|x|^{-\frac{2}{p-1}}\label{sub}
\ee 
for some $\delta<1$. Then any solution $u\in{\mathcal C}^2(\R^d\times[0,T])$ of problem \rf{nlh}--\rf{ini} with the property 
 $\lim_{|x|\to\infty}|x|^{\frac{2}{p-1}}u(x,t)=0$ satisfied uniformly on $[0,T]$,   exists globally in time and is bounded by $u_C$ 
\be
0\le u(x,t)< \delta u_C(x)\label{subb}
\ee 
for each $t>0$ and $x\in\R^d$. 
\end{theorem}

\begin{proof}
This result can be considered as a kind of comparison principle for equation \rf{nlh} when a subcritical solution $u$, $0\le u(x,t)<\delta u_C(x)$ (continuous off the origin with a proper decay at infinity) is compared with $\delta u_C$, and then this can be  continued  onto some interval $[T,T+h]$, and further, step by step, onto the whole half-line $[0,\infty)$. 
For analogous results in chemotaxis theory with either Brownian or fractional diffusion, see \cite{BKZ2,BKZ3}. 

We sketch the proof skipping some technical details related to the local existence of solutions with the assumed regularity and decay. 
Suppose {\em a contrario} that $u(x_0,t_0) =\delta u_C(x_0)$ for some $x_0\in\R^d$ with minimal $|x_0|$ and the least $t_0>0$. 
Consider the auxiliary  function 
\be
z(x,t)=|x|^{\frac{2}{p-1}}u(x,t). \label{z}
\ee
Under the {\em a contrario}  assumption $z(x_0,t_0)=\delta c$,  \  
and if $t_0$ is the first moment when $z(x,t)$ hits the level $\delta c\in\R$, we have 
\be 
\nabla z(x_0,t_0)=0,\ \ \Delta z(x_0,t_0)\le 0,\label{grad}
\ee 
since $x_0$ is the point where the maximum of $z$ (equal to $\delta c$) is attained. Let us compute 
\bea
\frac{\partial}{\partial t}z(x_0,t)\big|_{t=t_0}&=& 
|x_0|^{\frac{2}{p-1}}(\Delta u+u^p)\nonumber\\
&=&|x_0|^{\frac{2}{p-1}}\left(\Delta(|x|^{-\frac{2}{p-1}}z(x,t_0))\big|_{x=x_0} + |x_0|^{-\frac{2}{p-1}}z(x_0,t_0)^p\right)\nonumber\\
&=&\left(\Delta z(x_0,t_0)-c^{p-1}z(x_0,t_0) +z(x_0,t_0)^p\right)\nonumber\\ 
&\le& z(x_0,t_0)(z(x_0,t_0)^{p-1}-c^{p-1})<0\label{der}
\eea
since
\bea
\nabla(|x|^{-\g}z)&=&|x|^{-\g}\nabla z-\g|x|^{-\g-2}x\,z,\nonumber\\ 
\Delta(|x|^{-\g}z)&=&|x|^{-\g}\Delta z-2\g|x|^{-\g-2}x\cdot \nabla z+\g(\g+2-d)z,\nonumber
\eea
relations \rf{grad} hold, and recall \rf{c=} to see that $\g(\g+2-d)=-c^{p-1}$.  The inequality $\frac{\partial}{\partial t}z(x_0,t)\big|_{t=t_0}<0$   contradicts the assumption that $z$ hits for the first time the {\em constant} level $\delta c$ at $t=t_0$. 
\end{proof}

\begin{remark}\label{Noriko} 
This kind of result is not, of course,  new but the proof seems be somewhat novel. 
If $u_0$ is radially symmetric and $u_0(x)< \delta u_C(x)$ for some $\delta<1$,  then the solution of \rf{nlh}--\rf{ini} exists globally in time, see \cite[Theorem 1.1]{M1} and also \cite[Remark 3.1(iv)]{S}. Related results are in \cite[Lemma 2.2]{M2}, and stability of the singular solution is studied in \cite{PY1}.   
Results for not necessarily radial solutions starting either below or slightly above the singular solution $u_C$ are in \cite[Th. 10.4]{GV} (reported in \cite[Th. 20.5]{QS}),  and in \cite[Th. 1.1]{SW}. Note that solutions of the Cauchy problem \rf{nlh}--\rf{ini} in the latter case are nonunique. 
\end{remark} 

\bigskip

\noindent
 {\bf Solutions exploding in a finite time} 

\noindent
Our goal here is to show a finite time blowup, 
and that the critical size of some functional norm of initial data leading to a blowup is close to the optimal size  of this norm  guaranteeing the existence of a global in time regular positive solution. 
The first argument is essentially that of \cite{Fu}. 
The  considerations in \cite{GK1,GK2} employed some (quite complicated) moments and energies of solutions with Gaussian weights but not exactly quantity \rf{moment-G}.

    \begin{theorem}\label{blow}
Suppose that $u_0\ge 0$ satisfies the condition 
\be
\sup_{T>0}T^{\frac{1}{p-1}}\left\| {\rm e}^{T\Delta}u_0\right\|_\infty>  \left(\frac{1}{p-1}\right)^{\frac{1}{p-1}}.\label{Besov}
\ee 
Then, any local in time weak solution $u=u(x,t)$ of the Cauchy problem \rf{nlh}--\rf{ini} with $p>1$ cannot be continued beyond $t=T$. 
    \end{theorem}

\begin{proof}
Note that here  we assume merely $p>1$. 
For a fixed $T>0$ consider the weight function $G=G(x,t)$ which solves the backward heat equation with the  unit Dirac measure as the final time condition at $t=T$
\be 
G_t+\Delta G=0,\ \ \ G(.,T)=\delta_0.\label{G}
\ee 
Clearly, we have a solution 
\be
G(x,t)=(4\pi(T-t))^{-\frac{d}{2}}\exp\left(-\frac{|x|^2}{4(T-t)}\right),\label{G} 
\ee
the unique nonnegative one, satisfying moreover 
\be
\int_{\R^d} G(x,t)\dx=1\ \ {\rm for\ each}\ \ t\in[0,T).\label{L1}
\ee
We  consider $t\in[0,T)$, and define for a solution $u$ of \rf{nlh}--\rf{ini} which is supposed to  exist on $[0,T)$ the moment 
\be
W(t)=\int_{\R^d} G(x,t)u(x,t)\dx={\rm e}^{(T-t)\Delta}u(t)(0), \label{moment-G}
\ee
where ${\rm e}^{t\Delta}$, $t>0$,  denotes the heat semigroup on $\R^d$ defined with the Gauss--Weierstrass kernel.
Evidently, we have 
\be
W(0)\in[0,\infty)\ \ {\rm  for}\ \  u_0\ge 0, \ \ {\rm and}\ \ W(0)\in(0,\infty)\ \ {\rm if}\ \  0\le u_0\not\equiv 0\label{W-pos}
\ee and, moreover,  
\be
W(0)={\rm e}^{T\Delta}u_0(0).\label{W0}
\ee
Since $G$ decays exponentially in $x$ as $|x|\to\infty$, the moment $W$ is well defined (at least) for (weak, pointwise, distributional)  solutions $u=u(x,t)$ which are polynomially bounded in $x$ as $|x|\to\infty$. 
The evolution of the moment $W$ is governed by the identity 
\bea
\frac{{\rm d}W(t)}{\dt}&=& \int_{\R^d}(Gu_t+G_tu)\dx\nonumber\\
&=&\int_{\R^d} G(\Delta u+u^p)\dx-\int_{\R^d} \Delta G\, u\dx \nonumber\\
&=&\int_{\R^d} \Delta G\, u\dx+\int_{\R^d} Gu^p \dx-\int_{\R^d} \Delta G\, u\dx \nonumber\\
&=&\int_{\R^d} Gu^p\dx\nonumber\\
&\ge& \left(\int_{\R^d} Gu\dx\right)^p,\label{W-ev}
\eea 
where the third line means that $u$ is a weak solution, and the last line follows by the H\"older inequality and property \rf{L1}. 
Now, the differential inequality 
\be
\frac{{\rm d}W(t)}{\dt}\ge (W(t))^p\label{W-p}
\ee 
(with the strict inequality for nonconstant functions $u(t)$)
shows that $W(t)$ increases, and after integrating immediately leads to 
$$
W(0)^{1-p}-W(t)^{1-p}\ge (p-1)t
$$
for all $t\in[0,T)$, so that 
\be
W(t)\ge \left(W(0)^{1-p}-(p-1)t\right)^{-\frac{1}{p-1}}\label{Wt}
\ee 
again for all $t\in[0,T)$. 
Now, passing to the limit $t\nearrow T$,  it is clear that if 
\be
W(0)={\rm e}^{T\Delta}u_0(0)> \left((p-1)T\right)^{-\frac{1}{p-1}},\label{Wc}
\ee
then we arrive at a contradiction with   property \rf{W-pos}, so that the solution blows up not later than $t=T$. 
In other words, if the condition 
\be
T^{\frac{1}{p-1}}{\rm e}^{T\Delta}u_0(0)>  \left(\frac{1}{p-1}\right)^{\frac{1}{p-1}}
\label{blo}
\ee 
is satisfied, then the solution $u$ cannot exist for $t=T$.  
By the translational invariance of equation \rf{nlh}, for positive $u_0$ condition \rf{blo} is equivalent  to inequality \rf{Besov}. 
\end{proof}

\begin{remark}\label{big-data} 
Observe that for any positive initial condition $u_0\not\equiv 0$ there is $N>0$ such that \rf{Wc} is satisfied for $Nu_0$. 
Condition \rf{Wc} is also valid for each constant $u_0>0$ and  suitably large $T$.
\end{remark}

\begin{remark}\label{Fujita}
It is clear that if $d<\frac{2}{p-1}$, then each positive $u_0\not\equiv 0$ leads to a blowing up solution, as it has been proved  in \cite[Th. 1]{Fu}. Indeed, 
\bea
\sup_{T>0}T^{\frac{1}{p-1}}\|{\rm e}^{T\Delta}u_0\|_\infty&=&\lim_{T\to\infty}T^{\frac{1}{p-1}}\|{\rm e}^{T\Delta}u_0\|_\infty\nonumber\\
&=&\lim_{T\to\infty}T^{-\frac{d}{2}+{\frac{1}{p-1}}}(4\pi)^{-\frac{d}{2}}\int_{\R^d}{\rm e}^{-|x|^2/4T}u_0(x)\dx\nonumber\\
&=&\lim_{T\to\infty}T^{-\frac{d}{2}+{\frac{1}{p-1}}}(4\pi)^{-\frac{d}{2}}\|u_0\|_1=\infty,\nonumber
\eea
so a sufficient condition \rf{Besov} for blowup holds. 
Here, we offer a proof of the analogous result if $d=\frac{2}{p-1}$, an alternative to the one given in \cite[Th. 1]{W}. 
Since condition \rf{Besov} for $\frac{d}{2}=\frac{1}{p-1}$ reads \be
(4\pi)^{-\frac{d}{2}}\int_{\R^d}{\rm e}^{-|y|^2/4T}u_0(y)\dy>\left(\frac{1}{p-1}\right)^{\frac{1}{p-1}},\label{spec}
\ee
 and $\lim_{T\to\infty}W(0)=(4\pi)^{-\frac{d}{2}}\|u_0\|_1$, 
it suffices to show that $\|u(t)\|_1$ becomes  large enough for some $t\ge 0$, and replace the initial condition $u_0$ by $u(t)$.  

Problem \rf{nlh}--\rf{ini} for positive solutions can be rewritten in the mild form as 
\be
u(t)={\rm e}^{t\Delta}u_0+\int_0^t{\rm e}^{(t-\tau)\Delta}u(\tau)^p\dta.\label{Duh}
\ee 
Let us write the Duhamel formula \rf{Duh} for $t=T$ in a more detailed way as 
\bea
u(x,T)-{\rm e}^{T\Delta}u_0(x)&=& \int_0^T\int_{\R^d} (4\pi(T-\tau))^{-\frac{d}{2}}{\rm e}^{-|x-y|^2/4(T-\tau)}u(y,s)^p\dy\ \dta\nonumber\\ 
&\ge&\int_0^T\left(\int_{\R^d}(4\pi(T-\tau))^{-\frac{d}{2}}\frac12\left( {\rm e}^{-|x-y|^2/4(T-\tau)}+{\rm e}^{-|x+y|^2/4(T-\tau)}\right) u(y,\tau)\dy\right)^p\dta\label{2s}\\
&\ge&\int_0^T\left(\int_{\R^d}(4\pi(T-\tau))^{-\frac{d}{2}}{\rm e}^{-|x|^2/4(T-\tau)}{\rm e}^{-|y|^2/4(T-\tau)}u(y,\tau)\dy\right)^p\dta\nonumber
\eea
by the H\"older inequality and by the Cauchy inequality  $\frac12\left({\rm e}^{x\cdot y/2(T-\tau)}+{\rm e}^{-x\cdot y/2(T-\tau)}\right)\ge 1$. 
Now, integrating over $x\in\R^d$, we obtain as a consequence of \rf{2s}
\bea
\|u(T)\|_1&\ge&\|u_0\|_1+\int_0^T\int_{\R^d}{\rm e}^{-p|x|^2/4(T-\tau)}W(\tau)^p\dx\dta\nonumber\\ 
&\ge& \|u_0\|_1+W(0)^p\int_0^T\left(\frac{4\pi}{p}\right)^{\frac{d}{2}}(T-\tau)^{\frac{d}{2}}\dta\nonumber\\
&\ge&\|u_0\|_1+W(0)^pT^{\frac{d}{2}+1}\frac{2}{d+2}\left(\frac{4\pi}{p}\right)^{\frac{d}{2}}\nonumber\\
&=&\|u_0\|_1+\tilde c\left(T^{\frac{1}{p-1}}W(0)\right)^p\nonumber\\ 
&=&\|u_0\|_1+\tilde{\tilde c}\left(\int_{\R^d}{\rm e}^{-|y|^2/4T}u_0(y)\dy\right)^p\nonumber
\eea 
for some $\tilde{\tilde c}$ independent of $u_0$. Observe that the norm $\|u(t)\|_1$ increases in time. 
Therefore, by a shift of time, we have 
$$
\|u(t+T)\|_1\ge \|u(t)\|_1\left(1+\tilde{\tilde c}\|u_0\|^{p-1}\right)
$$
and it is clear that for some $t$ the norm $\|u(t)\|_1$ becomes large enough in order to condition \rf{Besov} holds with time shifted from $0$ to $t$, see also \rf{spec}.  
Therefore, $u(t)$ blows up in a finite time.  \qed
\end{remark}

\begin{remark}\label{Morrey} 
Here we discuss some questions related to applicability of Morrey spaces in the analysis of optimal conditions for local in time existence of solutions as well as for the finite time blowup of solutions. 
Recall that (homogeneous) Morrey spaces over $\R^d$ modeled on $L^q(\R^d)$, $q\ge 1$, are defined by their norms 
\be
|\!\!| u|\!\!|_{M^s_q}\equiv \left(\sup_{R>0,\, x\in\R^d}R^{d(q/s-1)} 
\int_{\{|y-x|<R\}}|u(y)|^q\dy\right)^{1/q}<\infty.\label{hMor}
\ee  

\noindent 
\emph{Caution:} the notation for Morrey spaces used elsewhere might be different, e.g. $M^s_q$ is  denoted by $M^{q,\lambda}$ with $\lambda=dq/s$ in \cite{S}. 

The most frequent situation is when $q=1$ and we consider $M^s_1\equiv M^s$. The spaces  $M^{d(p-1)/2}(\R^d)$ and $M^{d(p-1)/2}_q(\R^d)$, $q>1$, are  critical (i.e. invariant by scalings that conserve equation \rf{nlh}) in the study of equation \rf{nlh}, see \cite{S},  and we refer to \cite{B-SM,BKZ-NHM,BKZ3,Lem} for analogous examples in chemotaxis theory.  

For instance, if the norm $|\!\!| u_0|\!\!|_{M^{d(p-1)/2}_q}$  (for a number $q>1$) is small enough, then a solution of problem \rf{nlh}--\rf{ini} is global in time, see \cite[Proposition 6.1]{S} (and  for the chemotaxis system cf. \cite[Th. 1]{B-SM}).  
The former result can be proved directly (while the proof in \cite{S} was by contradiction) using the Picard iterations of the mapping 
$${\mathcal N}(u)(t)={\rm e}^{t\Delta}u_0+\int_0^t {\rm e}^{(t-\tau)\Delta}(|u|^{p-1}u)(\tau)\dta, 
$$
$u_{n+1}={\mathcal N}(u_n)$, $n=1,\,2,\, \dots$,   with $u_0\in M^{d(p-1)/2}_q(\R^d)$ ($q>1$, $p>1+2q/d$) small enough. They are convergent in the norm $\xn u(t)\xn = \sup_{t>0} t^{(1-q/r)/(p-1)}|\!\!| u(t)|\!\!|_{M^{rd(p-1)/(2q)}_r}$ for  $\max\{p,q\}<r<pq$, since $\xn {\rm e}^{t\Delta}u_0\xn <\infty$.  
\end{remark} 

\begin{remark} \label{Besov-Morrey}
Note that the quantity $\sup_{t>0}t^{\frac{1}{p-1}}\left\| {\rm e}^{t\Delta}u_0\right\|_\infty$ in \rf{Besov} is equivalent to the norm of the Besov space $B^{-\frac{2}{p-1}}_{\infty,\infty}(\R^d)$, e.g.  \cite[Remark 4.2]{S}. 
Moreover, for positive functions the quantity \rf{Besov} is equivalent to the Morrey space norm $M^{d(p-1)/2}(\R^d)$, see  \cite[Prop. 2 B)]{Lem}.  
\end{remark}

\begin{remark}\label{rate} 
Note that at the blowup time $T$ we have $\lim_{t\nearrow T}\|u(t)\|_\infty=\infty$ but some other norms --- in particular  $|\!\!| u(T)|\!\!|_{M^{d(p-1)/2}}$ --- can remain bounded, cf. also \cite[Remark 6.1(iv)]{S}. 
\end{remark} 

\bigskip

Taking into account the above remarks on the critical   Morrey space, we formulate the following partial dichotomy result 
\begin{corollary}\label{wniosek}
There exist two positive constants $c_d$ and $C_d$ such that if $p>1+\frac{2}{d}$ and for some  $q>1$  then 
\begin{itemize}
\item[(i)]
 $|\!\!| u_0|\!\!|_{M^{d(p-1)/2}_q}<c_d$ implies that problem \rf{nlh}--\rf{ini} has a global in time solution; 

\item[(ii)]
  $|\!\!| u_0|\!\!|_{M^{d(p-1)/2}}>C_d$ implies that each solution of problem \rf{nlh}--\rf{ini} blows up in a finite time.
\end{itemize}  
\end{corollary}

\bigskip

Below, we determine two threshold numbers measuring how big must be either $u_0$ compared to $u_C$ or the Morrey norm $M^{d(p-1)/2}(\R^d)$ of a radial $u_0$ in order to a solution of \rf{nlh}--\rf{ini} blows up in a finite time. 
\bigskip

\begin{theorem}\label{threshold}
(i) 
Let $d\ge 3$, $p>\frac{d}{d-2}$ (so that $d-\g>0$),  and define the threshold number 
\bea
{\mathcal N}&=&\inf\{N: {\rm  solution\ with\ the\ initial\ datum\ satisfying\ }u_0(x)\ge Nu_C(x) \nonumber\\  &\quad& {\rm blows\ up\ in\ a\ finite\ time}\}.\nonumber
\eea
Then, the relation  
$${\mathcal N}\to 1\ \ {\rm as\ \ } d\to\infty$$
holds. 
More precisely, 
\bea
1\ \ < \ \ {\mathcal N}&\le& 2^{\frac{1}{p-1}}\left(d-\frac{2p}{p-1}\right)^{-\frac{1}{p-1}}\frac{\Gamma\left(\frac{d}{2}\right)}{\Gamma\left(\frac{d}{2}-\frac{1}{p-1}\right)}\label{exactN}\\
&\approx& \left(1-\frac{2p}{d(p-1)}\right)^{-\frac{1}{p-1}}\label{asN}
\eea
as $d\to\infty$. 

\noindent (ii)
Moreover, there exists $\mathcal M>0$  such that if $u_0\ge 0$ is radially symmetric and such that $|\!\!| u_0\mn2 > {\mathcal M}$,  then the solution with $u_0$ as the initial data blows up in a finite time.   
A~rough estimate from above  for the threshold value of ${\mathcal M}$  when $p>\frac{d+2}{d}$ (i.e. exactly when $d-\g>0$) is the following 
\bea
{\mathcal M}&\le&\left(\frac{1}{p-1}\right)^{\frac{1}{p-1}}(4\pi)^{\frac{d}{2}}{\rm e}^{\frac{d-\g}{2}}(2(d-\g))^{\frac{\g-d}{2}}\label{exactM}\\
&\approx&
\left(\frac{1}{p-1}\right)^{\frac{1}{p-1}}\s_d\,\sqrt{\frac{\pi}{2}}\left(\frac{d}{2}-\frac{1}{p-1}\right)^{\frac{1}{p-1}+\frac12}\label{asM}
\eea
as $d\to\infty$. 
\end{theorem}

\begin{proof}
(i) We compute 
\bea 
t^{\frac{1}{p-1}}(4\pi t)^{-\frac{d}{2}}\int |x|^{-\g}\exp\left(-\frac{|x|^2}{4t}\right)\dx 
&=& 2^{-\g-1}\pi^{-\frac{d}{2}}\s_d\int_0^\infty {\rm e}^{-z}z^{\frac{d-\g}{2}-1}\,{\rm d}z \nonumber\\
&=& 2^{-\g}\frac{\Gamma\left(\frac{d-\g}{2}\right)}{\Gamma\left(\frac{d}{2}\right)}, \label{2asy}
\eea
where we used the fact  that  the area $\s_d$ of the unit sphere in $\R^d$ is equal to 
\be
\s_d=\frac{2\pi^{\frac{d}{2}}}{\Gamma\left(\frac{d}{2}\right)}, \label{pole}
\ee 
and we obtain  \rf{exactN}. The asymptotic formula in \rf{asN} is a  consequence of the Stirling formula for the Gamma function 
 \be
\Gamma(z+1)\approx \sqrt{2\pi z}\, z^z{\rm e}^{-z}\ \ \ {\rm and}\ \ \  \frac{\Gamma(z+a)}{\Gamma(z)}\approx z^a\ \ {\rm as}\ \ z\to\infty,\label{Stir}
\ee
see e.g. \cite{TE}. 
Indeed, 
 $${\mathcal N}\approx 2^{\frac{1}{p-1}}\left(d-\frac{2p}{p-1}\right)^{-\frac{1}{p-1}}\frac{\Gamma\left(\frac{d}{2}\right)}{\Gamma\left(\frac{d}{2}-\frac{1}{p-1}\right)}\approx 2^{\frac{1}{p-1}}\left(d-\frac{2p}{p-1}\right)^{-\frac{1}{p-1}}\left(\frac{d}{2}\right)^{\frac{1}{p-1}}$$ holds,
i.e. the discrepancy between the multiples of $|x|^{-\frac{2}{p-1}}$ in the initial data corresponding to global and blowing up solutions  
 converges   to $1$ as $d\to\infty$. 
\bigskip

(ii) Define for a radially symmetric function $u_0\ge 0$ its radial distribution function 
$$
M(r)=\s_d\int_0^r u_0(\r)\r^{d-1} \,{\rm d}\r,  
$$
so that $\s_du_0(r)r^{d-1}=M_r(r)$.  
It is not hard to see that for radial $u_0\ge 0$ we have $|\!\!|u_0\mn2=\sup_{R>0}R^{\g-d}\int_{\{|x|<R\}}u_0(x)\dx$. 
If $u_0\in M^{d(p-1)/2}(\R^d)$ with $|\!\!|u_0\mn2={\mathcal M}$, then for each $\e>0$ there exists $R>0$ such that $M(R)\ge ({\mathcal M}-\e)R^{d-\g}$.  
We need to estimate $T^{\frac{1}{p-1}}W(0)=T^{\frac{1}{p-1}}{\rm e}^{T\Delta}u_0(0)$ from below, and we do it in a rough way as
\bea 
T^{\frac{1}{p-1}}W(0)&=& T^{\frac{1}{p-1}}\int_0^\infty (4\pi T)^{-\frac{d}{2}}{\rm e}^{-r^2/4T}M_r(r)\dr \nonumber\\
&\ge & T^{\frac{1}{p-1}-\frac{d}{2}} (4\pi)^{-\frac{d}{2}}\int_0^R {\rm e}^{-r^2/4T}M_r(r)\dr \nonumber\\
&=& T^{\frac{1}{p-1}-\frac{d}{2}} (4\pi)^{-\frac{d}{2}}\left( {\rm e}^{-R^2/4T}M(R)+\int_0^R {\rm e}^{-r^2/4T}\frac{r}{2T}M(r)\dr \right)\nonumber\\
&\ge&T^{\frac{1}{p-1}-\frac{d}{2}} (4\pi)^{-\frac{d}{2}} ({\mathcal M}-\e)R^{d-\g}{\rm e}^{-R^2/4T}.\label{M1}
\eea
Since $\max_{T>0}T^{\frac{1}{p-1}-\frac{d}{2}}{\rm e}^{-R^2/4T}$ is attained when $\frac{R^2}{4T}=\frac12(d-\g)$ and this quantity is equal to $\left(\frac{R^2{\rm e}}{2(d-\g)}\right)^{\frac{\g-d}{2}}$, we need to determine when 
$$
(4\pi)^{-\frac{d}{2}}({\mathcal M}-\e)\left(\frac{2(d-\g)}{{\rm e}}\right)^{\frac{d-\g}{2}}
> \left(\frac{1}{p-1}\right)^{\frac{1}{p-1}}.
$$
This leads, by \rf{pole} and \rf{Stir},  to 
\bea 
{\mathcal M}&\gtrsim&2^\g\left(\frac{1}{p-1}\right)^{\frac{1}{p-1}}\frac{\s_d}{2}\frac{\Gamma\left(\frac{d}{2}\right)}{\Gamma\left(\frac{d-\g}{2}+1\right)}\sqrt{2\pi}\sqrt{\frac{d-\g}{2}}\nonumber\\ 
&\approx& 2^\g\left(\frac{1}{p-1}\right)^{\frac{1}{p-1}}\s_d \left(\frac{d-\g}{2}\right)^{\frac{\g-1}{2}}\sqrt{\frac{\pi}{2}}. 
\eea
Note that, compared to $|\!\!|u_C\mn2=\frac{c\s_d}{d-\g}$, the number  ${\mathcal M}$ is much bigger:
$$
 \frac{\mathcal M}{|\!\!|u_C\mn2}\approx \frac{\left(\frac{d-\g}{2}\right)^{\frac{\g+1}{2}}}{\left(\frac{d-p\g}{2}\right)^{\frac{\g}{2}}}\sqrt{2\pi} \approx\sqrt{\pi d},
$$
and, of course, ${\mathcal M}\le C_d$ by the definition of $C_d$ in Corollary \ref{wniosek} (ii). 
\end{proof}

\bigskip

\begin{remark}\label{dane}
It is well known, cf. \cite{LNi} and \cite[Th. 17.12]{QS}, that (under some supplementary assumptions) if 
$\lim_{|x|\to\infty}u_0(x)|x|^{\frac{2}{p-1}}=\infty$, then solutions of \rf{nlh}--\rf{ini} cannot be global in time. 
On the other hand, the authors of \cite{PY2} showed that for some  $u_0(x)\sim |x|^{-\frac{2}{p-1}}$ global in time solutions of \rf{nlh}--\rf{ini} exist and are unbounded as $t\to\infty$. 

A very short proof of the former result follows from the observation that such an  initial condition $u_0$ satisfies $|\!\!| u_0\mn2=\infty$, hence condition \rf{Besov} holds.  
Indeed, inequality
$$
\int_{\{|x|\le R\}}u_0(x)\dx\ge C\int_{R_0}^Rr^{-\g+d-1}\dr\approx CR^{d-\g}
$$ 
is satisfied with arbitrarily large constants $C$ and suitably big fixed $R_0$, so that 
\newline
 $\lim_{R\to\infty}R^{\g-d}\int _{|x|\le R}u_0(x)\dx \ge C$. 
\end{remark} 

\begin{remark}\label{nonexistence}
If  $u_0\ge 0$ is such that $\limsup_{t\searrow 0}t^{\frac{1}{p-1}}\left\| {\rm e}^{t\Delta}u_0\right\|_\infty>  \left(\frac{1}{p-1}\right)^{\frac{1}{p-1}}$, then the statement: \emph{ there exists $T>0$ and a solution of problem \rf{nlh}--\rf{ini} on $(0,T)$}  is not true. 
One is tempted to say that there is an instantaneous blowup of solution with such an initial data, but the correct statement is rather: \emph{ there is no continuity property of solutions with respect to (large) initial data in the space $M^{d(p-1)/2}(\R^d)$} when, e.g., the initial data with large $M^{d(p-1)/2}(\R^d)$ norms are truncated on levels growing to infinity. 
In other words: {\em for large initial data the problem \rf{nlh}--\rf{ini} is  ill-posed in ${\mathcal C}([0,T],M^{d(p-1)/2}(\R^d))$ for any $T>0$}. 
\end{remark}

\bigskip
\noindent 
{\bf Extensions and generalizations.}
The method of proving blowup with the use of the Gaussian moment \rf{moment-G} extends to other nonlinear equations with the heat operator in the principal part. 
For instance, the equation
$$
u_t=\Delta u+Q(x)|u|^{p-1}u
$$
considered in \cite{LX} leads to the following sufficient for blowup of its positive solutions 
$$
W(0)\ge (4\pi)^{-\frac{d}{2}}
\left(\int_0^T (p-1)(T-t)^{d(p-1)/2}\left(\int_{\R^d}Q(x)^{-\frac{1}{p-1}} {\rm e}^{-|x|^2/4(T-t)}\dx\right)^{1-p}\dt\right)^{-\frac{1}{p-1}}.$$ 
These integrals  can be calculated explicitly,   e.g. for $Q(x)=|x|^\beta$, $\beta\in\R$. 

Another example for which sufficient criteria for blowup can be easily  established with this method is the quasilinear equation 
$$
u_t=\Delta u +\nabla\cdot\left(x|u|^{p-1}u\right), 
$$
and, of course, radially symmetric solutions of the Keller-Segel system in chemotaxis theory as it was mentioned before.


\end{document}